\documentclass[a4paper]{amsart}
\usepackage{amssymb}

\newtheorem{theorem}{Theorem}[section]
\newtheorem{lemma}[theorem]{Lemma}
\newtheorem{cor}[theorem]{Corollary}
\newtheorem{corollary}[theorem]{Corollary}
\newtheorem{prop}[theorem]{Proposition}
\newtheorem{question}[theorem]{Question}

\theoremstyle{definition}

\theoremstyle{remark}
\newtheorem{remark}[theorem]{Remark}
\newtheorem{remarks}[theorem]{Remarks}
\newtheorem{convention}[theorem]{Convention}

\numberwithin{equation}{section}
\newcommand{\Q}{\mathbb{Q}}
\newcommand{\C}{\mathbb{C}}

\newcommand{\R}{\mathbb{R}}

\newcommand{\ev}{\mathrm{ev}}
\newcommand{\odd}{\mathrm{odd}}

\def\proofend{\hbox to 1em{\hss}\hfill $\blacksquare $\bigskip }

\def\Z{{\mathbb Z}}
\def\Zp #1{{\mathbb Z }/#1{\mathbb Z}}

\def\gammaneu{\kappa}

\DeclareMathOperator{\rank}{rk}

\title{Complete intersections with $S^1$-action}

\author{Anand Dessai}
\address{D\'epartement de Math\'ematiques, Universit\'e de Fribourg,  Chemin du Mus\'ee 23, 1700 Fribourg, Switzerland}
\email{anand.dessai@unifr.ch}
\thanks{}

\author{Michael Wiemeler}
\address{Institut f\"ur Mathematik, Universit\"at Augsburg, 86135 Augsburg, Germany}
\email{michael.wiemeler@math.uni-augsburg.de}
\thanks{The first author was supported by SNF Grants No. 200020-126795 and 200020-138147. The second author was supported by SNF Grant No. PBFRP2-133466 and DFG Grant HA 3160/6-1.}



\begin{document}
\begin{abstract}
{We give the diffeomorphism classification of complete intersections
with $S^1$-symmetry in dimension $\leq 6$. In particular, we show
that a $6$-dimensional complete intersection admits a smooth
non-trivial $S^1$-action if and only if it is diffeomorphic to the
complex projective space or the quadric. We also prove that in any odd complex dimension only finitely many complete intersections can carry a smooth effective action by a torus of rank $>1$.}
\end{abstract}

\maketitle

\noindent
\section{Introduction}
The use of algebraic and differential topology in the study of compact smooth transformation groups has a long history and many methods have been developed over the years. Often the strength/limitation of these methods can already be seen if one applies them to understand symmetries of manifolds of a particularly simple topological type, for example homotopy spheres or homotopy complex projective spaces. A prominent instance is a conjecture of Petrie \cite{Pe72} stating that among homotopy complex projective spaces only the ones with standard total Pontrjagin class can possibly admit a smooth non-trivial action by the circle $S^1$ (cf. for example \cite{DeTop} and references therein for methods and partial results confirming the conjecture). For a new result in this direction see Theorem \ref{Petrie theorem}, where the conjecture is proven in complex dimension $<12$ for rank two torus actions.

The main purpose of this paper is to study symmetries of complete intersections. In particular, we consider the following
\begin{question}\label{q1}
Which complete intersections admit a smooth non-trivial $S^1$-action?
\end{question}
This question bears resemblance to Petrie's conjecture although here the manifolds are not considered up to homotopy. In fact, the characteristic classes of a complete intersection live in a subring which is rationally isomorphic to the cohomology ring of a complex projective space and there is evidence that the Pontrjagin classes are relevant for the question above (see Theorem \ref{maintechnicaltheorem}).

Complete intersections play an important r\^ole in algebraic
geometry. According to a conjecture of Hartshorne \cite{Ha74} smooth subvarieties in projective space are of this kind provided their codimension is sufficiently small. In topology the classification of complete intersections up to diffeomorphism,
homeomorphism or homotopy equivalence has been an active research area for many
decades (see for example \cite{LiWo81, LiWo83}, \cite[Section
8]{Kr99}).

Many results are known about the symmetries of complete intersections viewed as complex manifolds. It is a classical fact that the automorphism group of a complete intersection is finite if its first Chern class is negative (cf. \cite[III.2]{Ko}). Moreover, as shown recently \cite[Theorem 3.1]{Be13} for complex dimension $\geq 2$ the automorphism group is zero-dimensional except for the quadric and the projective spaces. In other words only the homogeneous complete intersections admit a non-trivial circle
action preserving the complex structure. In contrast, only little seems to be known about {\em smooth} symmetries. Understanding the symmetries of a complete intersection (viewed as a smooth manifold) is a natural but difficult problem. We believe that these manifolds are good candidates to test the strength/limitation of methods from the theory of transformation groups.

In this paper we classify complete intersections with non-trivial smooth $S^1$-action in real dimension $\leq 6$. We also prove that in any odd complex dimension only finitely many complete intersections can carry a smooth effective action by a torus of rank $>1$.

Let us first recall that a complete intersection $X_n(d_1,\dots,d_r)\subset \C P^{n+r}$ is a
smooth $2n$-dimensional manifold given by a transversal intersection
of $r$ non-singular hypersurfaces in complex projective space. The
hypersurfaces are defined by homogeneous polynomials whose degrees
are given by an unordered $r$-tuple $d_1,\ldots ,d_r$. In general,
the induced complex structure of $X_n(d_1,\dots,d_r)$ depends on the choice
of the polynomials (cf. \cite{Hi54}, \S 2.1). However, as noted by Thom, the oriented diffeomorphism
type of a complete intersection only depends on $n$ and the
multi-degree $(d_1,\ldots ,d_r)$.

The two-dimensional complete intersections with $S^1$-symmetry are
diffeomorphic to the sphere or the torus and are given by
$X_1(1)\cong X_1(2)\cong S^2$ and $X_1(3)\cong X_1(2,2)\cong
S^1\times S^1$. This follows directly from the Lefschetz fixed point formula for the Euler characteristic and the classification of surfaces. Note that these complete intersections also admit holomorphic $S^1$-actions with respect to their natural complex structure (cf. \cite[Theorem 3.1]{Be13}).

In dimension four the classification of complete intersections with holomorphic $S^1$-symmetries does not coincide with the classification for smooth $S^1$-symmetries. Using Seiberg-Witten theory one can show the following

\begin{theorem}[Theorem \ref{4-theorem}]\label{intro 4-theorem}
A $4$-dimensional complete intersection $X_2(d_1,\dots,d_r)$ admits
a smooth non-trivial $S^1$-action if and only if
$X_2(d_1,\dots,d_r)$ is diffeomorphic to a complex projective plane
$X_2(1)$, a quadric $X_2(2)$, a cubic $X_2(3)$ or an intersection of
two quadrics $X_2(2,2)$.
\end{theorem}

Note that the $4$-dimensional complete intersections with smooth $S^1$-symmetries are precisely the ones with positive first Chern class. In dimension six we prove that only the homogeneous complete intersections admit $S^1$-symmetries.

\begin{theorem}\label{main theorem}
A $6$-dimensional complete intersection $X_3(d_1,\dots,d_r)$ admits
a smooth non-trivial $S^1$-action if and only if
$X_3(d_1,\dots,d_r)$ is diffeomorphic to the complex projective
space $X_3(1)$ or the quadric $X_3(2)$.
\end{theorem}
In particular, some $6$-dimensional complete intersections, like the
cubic $X_3(3)$ or the quartic $X_3(4)$, have positive first Chern
class but do not admit a smooth non-trivial $S^1$-action (for higher dimensional examples w.r.t. torus actions see Corollary \ref{Ricci Beispiele}). By Yau's solution \cite{Ya78} of the Calabi conjecture these manifolds also admit metrics of positive Ricci curvature.
This answers a question of Wilderich Tuschmann in the negative.
He asked the following variant of a problem of Yau: Does every manifold which admits a metric of positive Ricci curvature also admit a smooth non-trivial circle action?
Finding a counterexample to this question was the original motivation for our investigation. The problem of Yau  \cite[Problem 3, p. 671]{Ya82}, which asks whether a manifold of positive {\em sectional} curvature admits a smooth effective $S^1$-action, is still open.

Theorem \ref{main theorem} follows from a more general statement about
$6$-dimensional manifolds (see Theorem \ref{maintechnicaltheorem}) which we
prove using methods from equivariant cohomology and equivariant
index theory. For Hamiltonian circle actions Theorem \ref{main theorem} can also be
deduced from work of Tolman \cite{To10} on the classification
of Hamiltonian circle actions on symplectic $6$-dimensional manifolds with
$b_2=1$.

In higher dimensions only partial results towards a classification of complete intersections with
$S^1$-symmetry are known (at least to the authors). Examples with
$S^1$-symmetry, which come into mind, are of course the homogeneous ones, i.e. the complex projective
space and the quadric, which are diffeomorphic to
\(SU(n+1)/S(U(n)\times U(1))\) and \(SO(n+2)/SO(n)\times SO(2)\),
respectively. It is tempting to conjecture that, like in the complex setting, these are the only ones.

By a theorem of Atiyah and Hirzebruch
\cite{AtHi70} the index of the Dirac operator, the $\hat A$-genus,
vanishes on spin complete intersections with smooth non-trivial
$S^1$-action. In dimension $2n=4k$ the $\hat A$-genus of a spin complete
intersection $X_{n}(d_1,\dots,d_r)$ vanishes if and only if
$n+r+1-\sum _{j=1}^r d_j>0$, i.e. if the first Chern class is
positive. This was first shown by Brooks \cite{Br83} who gave an
explicit formula for the $\hat A$-genus in terms of $n$ and the
multi-degree $(d_1,\ldots ,d_r)$. In particular, the number of
diffeomorphism types of complete intersections with smooth
non-trivial $S^1$-action is finite if one restricts to spin complete
intersections and to a fixed even complex dimension. Also diffeomorphism finiteness is known for complete intersections in odd complex dimensions for nice $Pin(2)$-actions (cf.  also \cite[Th. 5.1]{De02} for a related result for special $S^3$-actions). Here the proof relies on vanishing results for indices of twisted $Spin^c$-Dirac operators and twisted elliptic genera given in \cite{De99}, \cite[Section 4]{DeTop}. As shown recently by the second author \cite{Wi15} this method can also be applied in the case of smooth effective torus actions provided the rank of the torus is larger than the second Betti number of the manifold. For complete intersections this gives the following finiteness theorem.

\begin{theorem}[Corollary \ref{finiteness corollary}]\label{rank two finiteness theorem}
  For each odd \(n\), there are, up to diffeomorphism, only finitely many complete intersections of complex dimension \(n\) which admit a smooth effective action of a \(2\)-dimensional torus.
\end{theorem}

More precise classification results have been obtained by Lev Kiwi for complex $5$-dimensional complete intersections with circle action. His PhD thesis contains an almost complete analysis of the possible fixed point data of circle actions in this dimension. From this analysis it follows that a complex $5$-dimensional complete intersection which admits a smooth effective action of a two-dimensional torus is diffeomorphic to the complex projective space or the quadric, i.e. is homogeneous (cf. \cite{KiDiss} for details).

The paper is structured as follows. In the next section we review basic properties of complete intersections. In Section \ref{low dim section} we explain
the aforesaid classification of $4$-dimensional complete
intersections with $S^1$-symmetry and derive Theorem \ref{main
theorem} from a more general theorem about certain $6$-dimensional manifolds (see Theorem
\ref{maintechnicaltheorem}). In Section \ref{localization section} we recall the Atiyah-Bott-Berline-Vergne localization formula and derive a structural result for manifolds whose even degree rational cohomology subring is like the one of a complex projective space (see Proposition \ref{structure prop}). Section \ref{section prelim} contains some preliminary facts about $6$-dimensional manifolds with $S^1$-symmetry. The proof of
Theorem \ref{maintechnicaltheorem} consists of a case by case study
of the possible $S^1$-fixed point configurations which is carried out in Section \ref{sectionmaintechnicaltheorem}. In Section \ref{finiteness section} we first discuss vanishing theorems for twisted elliptic genera of $Spin^c$-manifolds with torus action and then prove Theorem \ref{rank two finiteness theorem} as well as a special case of the Petrie conjecture (see Theorem \ref{Petrie theorem}). In the appendix we collect formulas from
equivariant cohomology and equivariant index theory which are used
in the proof.
\bigskip

The first author likes to thank Volker Puppe for stimulating discussions on the subject and on possible generalization to continuous $S^1$-actions. The second author wants to thank Nigel Ray and the University of Manchester for hospitality while he was working on this paper.
We also want to thank Daniel Loughran for giving us the reference \cite{Be13} and Matthias Franz for helpful comments on an earlier version of the paper.

\section{Basic properties of complete intersections}
In this section we review relevant topological properties of complete intersections. Let $M:=X_n(d_1,\dots,d_r)$ be a complete intersection given by a
transversal intersection of $r$ non-singular hypersurfaces in $\C
P^{n+r}$ of degree $d_1,\ldots ,d_r$. The oriented diffeomorphism type
of $M$ only depends on $n$ and the multi-degree $(d_1,\dots,d_r)$. Note that
we may always assume $d_j\geq 2 $ if $r\geq 2$. In fact, up to
diffeomorphism, intersection with hypersurfaces of degree one
amounts to cutting down the dimension of the ambient complex
projective space.

Let $\gamma $ denote the restriction of the dual Hopf bundle over
$\C P^{n+r}$ to $M$ and let $x:=c_1(\gamma )\in H^2(M;\Z )$. For
later reference we collect some properties of $M$ which follow from
the Lefschetz hyperplane theorem, Poincar\'e duality and properties
of characteristic classes.

\begin{prop}\label{compinterprop}
\begin{enumerate}
\item $M$ is simply connected for $n>1$.
\item $H^*(M;\Z )$ is torsion-free. The cohomology groups of $M$ and $\C P^n$ are isomorphic outside the middle dimension, i.e. $H^i(M;\Z)\cong H^i(\C P^n;\Z )$ for $i\neq n$. Moreover $H^{2i}(M;\Z)$ is generated by $x^i$ for $2i<n$ .
\item $[x^n]_M=\prod_j d_j$, where $[\quad ]_M$ denotes evaluation on the fundamental cycle.
\item The total Chern class of $M$ is given by
$$c(M)=(1+x)^{n+r+1}\cdot \prod _{j=1}^r(1+d_j\cdot x)^{-1}.$$
In particular, $c_1(M)=(n+r+1-\sum _j d_j)\cdot x$.
\item The total Pontrjagin class of $M$ is given by
$$p(M)=(1+x^2)^{n+r+1}\cdot
\prod
_{j=1}^r(1+d_j^2\cdot x^2)^{-1}.$$
In particular, $p_1(M)=(n+r+1-\sum _j d_j^2)\cdot x^2$.
\item The Euler characteristic of $M$, $\chi (M)$, is equal to $[c_n(M)]_M$.
 For $n=1$, $\chi (M)=d_1\cdot \ldots \cdot d_r\cdot (2-\sum _{j=1}^r(d_j-1))$.
  For $n=3$, $\chi (M)<0$ except for $M=X_3(1)=\C P^3$ and $M=X_3(2)=SO(5)/(SO(3)\times SO(2))$
   which have Euler characteristic equal to $4$.
\end{enumerate}
\end{prop}
For a proof of these properties see for example \cite{Hi54}. The
inequality for the Euler characteristic stated in (6) may be deduced
from \cite[formula (5), p. 465]{Hi54}, see also~\cite{EwMo76}.

If a complete intersection comes with an action by a torus $T$ one can consider the Serre spectral sequence for \(H^*(M_T;\mathbb{Q})\), where $M_T:=ET\times _{T} M$ is the Borel construction. It turns out that the spectral sequence degenerates at the \(E_2\)-level. We will explain this in the following more general situation:

Let $M$ be a $2n$-dimensional closed oriented manifold, $n\geq 2$. If $n$ is even we assume that $H^*(M;\mathbb{Q})$ is concentrated in even degrees. If $n$ is odd we assume $b_1(M)=0$ and that the rational cohomology in even degrees, $H^{\ev}(M;\mathbb{Q})$, is generated by elements of degree $2$. Note that these assumptions are satisfied by complete intersections. Suppose a torus $T$ acts on $M$. Then we have

\begin{lemma}\label{specseq lemma}
 The Serre spectral sequence for \(H^*(M_T;\mathbb{Q})\) degenerates at the \(E_2\)-level, i.e. \(M\) is equivariantly formal.
\end{lemma}
\begin{proof} If \(n\) is even the cohomology of \(M\) is concentrated in even degrees. Hence, the Serre spectral sequence degenerates in this case.

Therefore assume that \(n\) is odd. Since \(b_1(M)=0\), all differentials \(d_r\) vanish on \(E_r^{*,2}\). Since $H^{\ev}(M;\mathbb{Q})$ is generated by elements of degree $2$ it follows by the multiplicativity of \(d_r\) that all differentials $d_r$ vanish on \(E_r^{*,2*}\). Suppose the differentials $d_{s}$, $s<r$, vanish on \(E_s^{*,*}\). To conclude that this also holds for $d_r$ we need to show that the image of an element $y\in E_r^{*,2*+1}$ under $d_r$ is zero. Looking at homogeneous parts we may assume that $y\in E_r^{*,2n-2k+1}$. By dimensional reasons we have \(E_2^{*,2n+1}=E_\infty^{*,2n+1}=0\). Hence, it follows from the multiplicativity of the differential that for classes $x_i \in E_r^{0,2}\cong  H^2(M;\Q)$, $i=1,\ldots ,k$, one always has $d_r(y)\cdot \prod_i x_i=d_r(y\cdot\prod _i x_i)=0$. Note that $E_r^{*,*}=E_2^{*,*}\cong H^*(BS^1;\Q )\otimes H^*(M;\Q )$. Therefore, by Poincar\'e duality, it follows that $d_r(y)=0$. Hence, the spectral sequence degenerates at the \(E_2\)-level.
\end{proof}

\section{$S^1$-actions on low dimensional complete intersections}\label{low dim section}
In this section we discuss Question \ref{q1} in more detail for
complete intersections in low dimensions. We give a proof of Theorem \ref{intro 4-theorem} using Seiberg-Witten theory. We also state a theorem about certain $6$-dimensional manifolds and
apply it to obtain the classification of $6$-dimensional
complete intersections with $S^1$-symmetry mentioned in the
introduction.

Let us start with the discussion for surfaces. By the classical Lefschetz fixed point formula the Euler
characteristic of a manifold with $S^1$-action is equal to the Euler
characteristic of the $S^1$-fixed point manifold. Combining this
with the classification of surfaces it follows that the only
orientable closed $2$-manifolds with $S^1$-symmetry are the sphere $S^2$
and the torus $S^1\times S^1$. Applying the formula for the Euler
characteristic given in Proposition \ref{compinterprop} (6) one
finds that among $2$-dimensional complete intersections only
$X_1(1)\cong X_1(2)\cong S^2$ and the elliptic curves $X_1(3)\cong X_1(2,2)\cong
S^1\times S^1$ admit a smooth non-trivial $S^1$-action.

In dimension four Seiberg-Witten theory leads to the following classification, probably well-known to the experts.
Since we couldn't
find a proof in the literature, we give
an argument below.
\begin{theorem}\label{4-theorem}
A $4$-dimensional complete intersection $X_2(d_1,\dots,d_r)$ admits
a smooth non-trivial $S^1$-action if and only if
$X_2(d_1,\dots,d_r)$ is diffeomorphic to a complex projective plane
$X_2(1)$, a quadric $X_2(2)$, a cubic $X_2(3)$ or an intersection of
two quadrics $X_2(2,2)$.
\end{theorem}

\begin{proof} We first explain why $X_2(1)$, $X_2(2)$, $X_2(3)$ and $X_2(2,2)$ admit
 a smooth non-trivial $S^1$-action.  For the complex projective space
$X_2(1)$ and the quadric $X_2(2)$, which are homogeneous, this is
obvious. One knows that $X_2(3)$ (resp. $X_2(2,2)$) is obtained by
blowing up $\C P^2$ at $6$ (resp. $5$) points in general position
(cf. \cite[page 653]{LiWo81}, \cite{Ma86}, \cite[Section
3.5]{Ko04}). Hence, $X_2(3)\cong \C P^2 \sharp 6 \overline {\C P^2}$
and $X_2(2,2)\cong \C P^2 \sharp 5 \overline {\C P^2}$. Since any
connected sum of $\C P^2$'s and $\overline {\C P^2}$'s admits a
smooth non-trivial $S^1$-action, the same holds for $X_2(3)$ and
$X_2(2,2)$.

To show that no other complete intersection admits a smooth
non-trivial $S^1$-action we combine certain facts about
Seiberg-Witten invariants. For any $4$-di\-mensional complete
intersection $M$ different from $X_2(1)$, $X_2(2)$, $X_2(3)$ and
$X_2(2,2)$ one knows that $b_2^+(M)$ (the dimension of the maximal
subspace of $H^2(M)$ on which the intersection form is positive
definite) is greater than one (cf. \cite[page 650]{LiWo81}). Since
$M$ is K\"ahler with $b_2^+(M)>1$ the Seiberg-Witten invariant for
$M$ with its preferred $Spin^c$-structure is $\pm 1$ by the
pioneering work of Witten \cite{Wi94}. On the other hand, Baldridge
showed in \cite{Ba04} that for any smooth closed $4$-manifold with
$b_2^+>1$ the Seiberg-Witten invariant vanishes if the manifold
admits a circle action with fixed point. Since $\chi(M)>0$ any
$S^1$-action on $M$ must have a fixed point. Hence, $M$ does not
admit a smooth non-trivial $S^1$-action.
\end{proof}

\begin{remarks}\begin{enumerate}
\item By Freedman's classification \cite{Fr82} of simply connected topological $4$-dimensional manifolds and the classification of indefinite odd forms any non-spin complete intersection $X_2(d_1,\dots,d_r)$ is homeomorphic to a connected sum of $\C P^2$'s and $\overline {\C P^2}$'s and, hence, admits a {\em continuous} non-trivial $S^1$-action.
\item The only spin complete intersection with {\em smooth} non-trivial $S^1$-action in dimension $4$ is the quadric $X_2(2)$. This follows directly from the $\hat A$-vanishing theorem of Atiyah-Hirzebruch \cite{AtHi70} and the formula for the first Pontrjagin class given in Proposition \ref{compinterprop} (5). The $\hat A$-vanishing theorem does not apply in general to {\em continuous} $S^1$-actions. However, it is known that the signature vanishes on a $4$-dimensional spin complete intersection with {\em continuous} $S^1$-action provided the involution in $S^1$ acts non-trivially and locally smoothly \cite{Ru95} or the number of orbit types near every $S^1$-fixed point is at most four \cite{HuPu98}. It follows that among spin complete intersections the quadric is the only one with such an action.
 \end{enumerate}
\end{remarks}

We now come to the classification of complete intersections with smooth
$S^1$-action in dimension $6$ stated as Theorem \ref{main theorem} in the introduction. This result is a
consequence of the following theorem which will be proved in Section \ref{sectionmaintechnicaltheorem}.

\begin{theorem}\label{maintechnicaltheorem}
Let $M$ be a smooth oriented closed $6$-dimensional manifold with
torsion-free homology, $b_1(M)=0$, $H^2(M;\mathbb{Z})=\langle
x\rangle$, \(p_1(M)=\rho \cdot x^2\) with \(\rho \leq 0\),
\(x^3\neq0\) and $\chi(M)<4$. Then $M$ does not support a
smooth non-trivial circle action.
\end{theorem}

\begin{proof}[Proof of Theorem \ref{main theorem}] The complex projective space
 $X_3(1)$ and the quadric $X_3(2)$ are homogeneous and, hence, admit a smooth non-trivial $S^1$-action.

Now assume $M:=X_3(d_1,\ldots ,d_r)$ is different from $X_3(1)$ and
$X_3(2)$. Then, either $r=1$ and $d_1\geq 3$ or $r\geq 2$ and
$d_j\geq 2\; \forall \, j$. In view of Proposition
\ref{compinterprop} $M$ satisfies all the conditions of Theorem
\ref{maintechnicaltheorem} and, hence, $M$ does not admit a smooth
non-trivial $S^1$-action.\end{proof}

\begin{remark} Note that there are examples with nontrivial circle action and \(\rho=1\) which satisfy all the other conditions in Theorem~\ref{maintechnicaltheorem}. They can be constructed as follows.

There are linear \(S^1\)-actions on \(S^3\) which have
one-dimensional fixed point sets. For a fixed point
the isotropy representation at this point is completely arbitrary. By
taking products of such actions we get an action of \(S^1\) on
\(S^3\times S^3\) with a two-dimensional fixed point set and
arbitrary isotropy representations at the fixed points.

Moreover, we may restrict the action of \(SO(5)\) on the complete
intersection \(X_3(2)=SO(5)/SO(3)\times SO(2)\) to a subgroup of
\(SO(5)\) isomorphic to \(S^1\) such that \(X_3(2)^{S^1}\) has a
two-dimensional component.

Therefore we may form the \(S^1\)-equivariant connected sum of
\(S^3\times S^3\) and \(X_3(2)\). This connected sum satisfies all
the assumptions of Theorem \ref{maintechnicaltheorem} except that
the first Pontrjagin class \(p_1((S^3\times S^3)\sharp X_3(2))\) is
equal to \(x^2\).\end{remark}

\begin{remark} For a $6$-dimensional manifold satisfying the cohomological assumptions in Theorem~\ref{maintechnicaltheorem} there are, by surgery theory, infinitely many pairwise non-diffeomorphic smooth manifolds inside its homotopy type with nonpositive first Pontrjagin class (i.e. $p_1=\rho \cdot x^2$ with $\rho \leq 0$). By Theorem~\ref{maintechnicaltheorem} none of these admit a smooth non-trivial $S^1$-action.
\end{remark}


\section{Atiyah-Bott-Berline-Vergne localization formula}\label{localization section}

In this section we recall the Atiyah-Bott-Berline-Vergne localization formula and illustrate its strength by a number of applications which will be used in the subsequent sections. For more information on the localization formula as well as some explicit formulas we refer to the appendix A.1. In the following we will restrict to smooth $S^1$-actions. However, many of the results also carry over to more general situations. In particular, Proposition \ref{structure prop}  below can be extended directly to torus actions.

Let $M$ be a smooth closed oriented connected $m$-dimensional manifold with smooth $S^1$-action and $M_{S^1}:=ES^1\times _{S^1} M$ the Borel construction. Let $[\quad ]_M:H^m(M;\Z )\to H^{0}(pt;\Z )\cong \Z$ denote evaluation on the fundamental cycle of $M$ and let $[\quad ]_M$ also denote the integration over the fiber map $H^m(M_{S^1};\Z )\to H^{0}(pt_{S^1};\Z )=H^{0}(B{S^1};\Z )\cong \Z$.

Let $y\in H^m(M;\Z )$ and suppose $\bar y\in H^{m}(M_{S^1};\Z )$ is an equivariant cohomology class which restricts to $y$ under $H^*(M_{S^1};\Z )\to H^*(M;\Z )$. Note that $[y]_M=[\bar y]_M$ since $\deg y=\dim M$. By the localization formula in equivariant cohomology of Atiyah-Bott and Berline-Vergne \cite{BeVe82, AtBo84} the equivariant class $[\bar y]_M$ can be computed in terms of local data at $M^{S^1}$:

\begin{equation}\label{loc formula} [y]_M=[\bar y]_M=\sum _{Z\subset M^{S^1}} \mu
(\bar y,Z).\end{equation}
Here the sum runs over the connected components $Z$ of $M^{S^1}$ (with fixed orientation) and the local datum $\mu
(\bar y,Z)$ at $Z$ is given by
$$\mu
(\bar y,Z)=[(\bar y_{\vert Z})\cdot e_{S^1}(\nu _Z)^{-1}]_Z,$$
where $\bar y_{\vert Z}$ is the restriction of $\bar y$ to $Z$ and $e_{S^1}(\nu _Z)$ is the Euler class of the equivariant normal bundle of $Z\subset M$.

In the following we will assume that $x\in H^2(M;\Z )$ is a class which can be lifted to an equivariant class $\bar x\in H^2(M _{S^1};\Z )$. A simple spectral sequence argument shows that for $b_1(M)=0$ this is always the case. In more geometric terms this means that the $S^1$-action lifts to the complex line bundle $\gamma$ over $M$ with $c_1(\gamma )=x$ and $\bar x$ is the first Chern class of the equivariant line bundle \cite[Corollary 1.2]{HaYo76}. Moreover the choice of $\bar x$ corresponds to the choice of the lift of the $S^1$-action to $\gamma$.

The restriction of $\bar x$ to a connected component $Z$ of $M^{S^1}$ takes the form $x_{\vert Z}+a_Z\cdot z$ where $z\in H^2(BS^1;\Z)$ is the preferred generator and $a_Z\in \Z $ is the weight of the $S^1$-representation given by restricting $\gamma $ to a point in $Z$. Hence, the $S^1$-equivariant first Chern class $\bar x$ at the connected components is given by $\{x_{\vert Z}+a_Z\cdot
z\, \mid \, Z\subset M^{S^1}\}$. For any $l\in \Z$ we can choose as a lift the class $\bar x +l\cdot z$ or, more geometrically, we can change the $S^1$-action on $\gamma $ by tensoring the line bundle with the complex one dimensional representation with weight $l$. Note that for the new lift the restriction to the connected components is given by $\{x_{\vert Z}+(a_Z+l)\cdot
z\, \mid \, Z\subset M^{S^1}\}$. In particular, we can choose for a component $Z$ a lift $\bar x_Z\in H^2(M _{S^1};\Z )$ of $x$ such that the restriction of $\bar x_Z$ to $Z$ is equal to $x_{\vert Z}$ (i.e. a fiber of $\gamma $ over a point of $Z$ is a trivial $S^1$-representation). For later use let us point out the following lemma which follows directly.

\begin{lemma}\label{lemma 4.1} Let $y\in H^m(M;\Z )$ be divisible by $x^r$ and let $\bar y\in H^{m}(M_{S^1};\Z )$ be an equivariant lift of $y$ which is divisible by $\bar x_Z^r$. Suppose $x_{\vert Z}^r$ vanishes in $H^*(Z;\Q )$. Then the local datum $\mu (\bar y,Z)$ in (\ref{loc formula}) vanishes.\proofend
\end{lemma}

Next we apply the foregoing to the situation where $M$ is of even dimension,  $m=2n$, and $[x^n]_M\neq 0$. Let $Z_i$, $i=1,\ldots ,k$, denote the connected components of $M^{S^1}$ and let $r_i\geq 0$ be such that $x_{\vert Z_i}^{r_i}\neq 0$ and $x_{\vert Z_i}^{r_i+1}=0$ in $H^*(Z;\Q )$. Then we have

\begin{lemma} $\sum _i (r_i+1)\geq n+1$.\end{lemma}

\begin{proof} Let us assume to the contrary that $\sum _i (r_i+1)\leq n$. Then we can choose lifts $\xi_j $, $j=1,\ldots ,n$, of $x$ such that for each $i$ at least $r_i+1$ of the $\xi_j$'s have the property that there restriction to $Z_i$ is equal to $x_{\vert Z_i}$.

Let $\bar y:=\prod_j \xi _j \in H^{2n}(M_{S^1};\Z )$. By the previous lemma the local datum $\mu (\bar y,Z_i)$ of $[\bar y]_M$ at $Z_i$ vanishes for every $i$. This contradicts $[\bar y]_M=[x^n]_M\neq 0$ in view of $(\ref{loc formula})$.\end{proof}

In terms of Betti numbers the last lemma says that the sum of even Betti numbers  of $M^{S^1}$, $b_{\ev} (M^{S^1})$, is at least $n+1$ (note that $b_{\ev}(M^{S^1})= \sum _i b_{\ev}(Z_i)$ and $b_{\ev}(Z_i)\geq r_i+1$ for trivial reasons). In the case $b_{\ev} (M)=n+1$ one obtains the following structural result.

\begin{prop}\label{structure prop}  Let \(M\) be a smooth \(S^1\)-manifold of dimension \(m=2n\) with $b_1(M)=0$ such that
  \begin{equation*}
    H^{\ev}(M;\mathbb{Q})=\mathbb{Q}[x]/(x^{n+1})\cong H^{\ev}(\C P^n; \mathbb{Q})
  \end{equation*}
as algebras with \(\deg x = 2\)  and let \(\bar{x}\in H^2(M_{S^1};\mathbb{Q})\) be a lift of \(x\). Let \(Z_1,\dots, Z_k\) be the components of \(M^{S^1}\) and  \(n_i=\frac{1}{2}\dim Z_i\). For \(i=1,\dots,k\) let \(pt_i\in Z_i\) and \(a_i=\bar{x}|_{pt_i}\in H^2(BS^1;\mathbb{Q})\). Then we have:
\begin{enumerate}
\item \(H^{\ev}(Z_i;\mathbb{Q})\cong H^{\ev}(\C P^{n_i};\mathbb{Q})\),
\item $x$ restricts to a generator of $H^2(Z_i;\Q )$,
\item \(\sum_i (n_i+1)=n+1\).
\item The \(a_i\) are pairwise distinct.
\end{enumerate}
\end{prop}

Note that by Proposition \ref{compinterprop} the assumptions on the cohomology ring are satisfied by any complete intersection of odd complex dimension.
The proposition can be shown by adapting the classical arguments for cohomology complex projective spaces with circle action (cf. e.g. \cite[VII, Th. 5.1]{Br72} or \cite[Th. IV.3]{Hs75}) to the situation above. Here we will give a proof based on the localization formula (\ref{loc formula}).

\bigskip
\begin{proof} Let us recall that one always has $b_{\ev}(M^{S^1})\leq b_{\ev}(M)$. This follows from an inspection of the spectral sequence for the Borel construction $M_{S^1}\to BS^1$ (cf. for example  \cite[VII, Th. 2.2]{Br72}). Here we have $b_{\ev}(M^{S^1})=b_{\ev}(M)=n+1$ since by Lemma \ref{specseq lemma} the spectral sequence degenerates (cf. \cite[VII, Th. 1.6]{Br72}). By the last lemma, $b_{\ev}(M^{S^1})=\sum _i b_{\ev}(Z_i)\geq \sum _i (r_i+1)\geq  n+1=b_{\ev}(M)$. Thus we get $r_i=n_i$, $b_{\ev}(Z_i)=n_i+1$ and  $H^{\ev}(Z_i;\mathbb{Q})\cong \mathbb{Q}[x_{\vert Z_i}]/(x_{\vert Z_i}^{n_i+1}))\cong H^{\ev}(\C P^{n_i};\Q )$. This proves the first three statements. For the last statement assume to the contrary that there exist components $Z_s\neq Z_t$ with $a_s=a_t$. For the other components $Z_i$, $i\not\in \{s,t\}$, let us choose lifts $\xi_j $, $j=1,\ldots , n-n_s-n_t-1$, of $x$ such that for each $i$ at least $n_i+1$ of the $\xi_j$'s have the property that there restriction to $Z_i$ is equal to $x_{\vert Z_i}$. Let $\xi$ be the lift of $x$ such that the restriction to $Z_s$ and $Z_t$ is equal to $x_{\vert Z_s}$ and $x_{\vert Z_t}$, respectively. Applying the localization formula to $\xi ^{n_s+n_t+1}\prod_j \xi_j $ we get using Lemma \ref{lemma 4.1} the contradiction $[x^n]_M=0$. Hence, the \(a_i\) are pairwise distinct.\end{proof}
\section{Preliminaries for the proof of Theorem \ref{maintechnicaltheorem}}\label{section prelim}

Let $M$ be a smooth orientable closed $6$-dimensional manifold with torsion-free homology,
$b_1(M)=0$ and $b_2(M)=1$.

Let $x$ be a generator of $H^2(M;\Z )$. The manifolds we are
interested in fulfill the following conditions:

\begin{enumerate}
\item The Euler characteristic of $M$ satisfies $\chi (M)<4$.
\item $p_1(M)=\rho \cdot x^2$ with $\rho \leq 0$.
\item $x^3\neq 0$.
\end{enumerate}
We fix the orientation of $M$ such that $t:=[x^3]_M>0$. Here $[\quad ]_N$ denotes, as before, evaluation on the
fundamental cycle of an oriented closed manifold $N$.

Note that $b_3(M)$ is even, since the intersection form  is
skew-symmetric. If $M$ is simply connected, then by the structure
result of Wall \cite{Wa66} $M$ is diffeomorphic to the connected sum
of a twisted complex projective space (with twist number $t$) and
$b_3(M)/2$ copies of $S^3\times S^3$.

We now assume that $M$ admits a smooth effective $S^1$-action. In the following we identify $\Zp m$ with the cyclic subgroup of order $m$ in $S^1$. We use the shorthand
notation $b_{\odd}=\sum _{2k+1}b_{2k+1}$ and $b_{\ev}=\sum
_{2k}b_{2k}$ for the odd and even Betti numbers, respectively.

\begin{prop}\label{bettiprop} $b_{\ev}(M^{S^1})=b_{\ev}(M)=4$. For a prime $p$ and \(l>0\),
$$b_{\ev}(M^{\Zp {p^l}})=\rank \, H^{\ev}(M^{\Zp {p^l}};\Zp p)=\rank \, H^{\ev}(M;\Zp p)=4.$$
\end{prop}
The following type of argument is well-known and only included for completeness and for the convenience of the reader.
\begin{proof} By Lemma \ref{specseq lemma} the spectral sequence for $M\hookrightarrow
M_{S^1}\to BS^1$ degenerates at the $E_2$-level which implies $b_{\ev}(M^{S^1})=b_{\ev}(M)=4$ (cf. \cite[Th. 1.6, Th. 2.1, p. 374-375]{Br72} or apply Proposition \ref{structure prop}).

Note that $S^1$ acts on $M^{\Zp {p^l}}$ and $(M^{\Zp {p^l}})^{S^1}=M^{S^1}$.
We fix a large prime $q$ such that the action of $\Zp q\subset S^1$
satisfies $M^{\Zp q}=M^{S^1}$, $b_i(M)=\rank \, H^{i}(M;\Zp q)$,
$b_i(M^{\Zp {p^l}})=\rank \, H^{i}(M^{\Zp {p^l}};\Zp q)$ and
$b_i(M^{S^1})=\rank \, H^{i}(M^{S^1};\Zp q)$.

Recall from \cite[Th. 2.2, p. 376-377]{Br72} that for any smooth
$S^1$-manifold $Z$ and any prime $p$ one has the inequality $\rank
\, H^{\ev}(Z^{\Zp p};\Zp p)\leq \rank \, H^{\ev}(Z;\Zp p)$.
By induction we have
$\rank
\, H^{\ev}(Z^{\Zp {p^l}};\Zp p)\leq \rank \, H^{\ev}(Z;\Zp p)$.
 Recall
also that $b_i(X)\leq  \rank \, H^{i}(X;\Zp p)$ for any space $X$.
As an application of these properties one obtains
$$b_{\ev}(M^{S^1})=\rank \, H^{\ev}(M^{S^1};\Zp q )=\rank \, H^{\ev}((M^{\Zp {p^l}})^{\Zp q};\Zp q )$$
$$\leq \rank \,  H^{\ev}(M^{\Zp {p^l}};\Zp q )=b_{\ev}(M^{\Zp {p^l}})$$
$$\leq \rank \,  H^{\ev}(M^{\Zp {p^l}};\Zp p )\leq \rank \,  H^{\ev}(M;\Zp p ).$$
Since the homology of $M$ is torsion-free $\rank \,H^{\ev}(M;\Zp p
)=\rank \,H^{\ev}(M;\Z)=b_{\ev}(M)=4$. Since $b_{\ev}(M^{S^1})=4$ all
inequalities in the display formula above are in fact equalities. In
particular,
$$b_{\ev}(M^{\Zp {p^l}})=\rank \, H^{\ev}(M^{\Zp {p^l}};\Zp p)=\rank \, H^{\ev}(M;\Zp p)=4.$$
\end{proof}

\bigskip
\noindent
\begin{cor}\label{orientcor}
For any prime $p$ and \(l>0\) the fixed point manifold $M^{\Zp {p^l}}$ is
orientable.
\end{cor}

\begin{proof} If $p$ is an odd prime this is trivial since the action of $\Zp {p^l}$
 on the normal bundle of the fixed point manifold induces a complex structure.
 If $p=2$ the claim follows from Proposition \ref{bettiprop}. Let $X$ be a connected component of $M^{\Zp
 {2^l}}$. Since $\Zp {2^l}$ acts orientation
 preserving $X$
 is even-dimensional, say $\dim X= 2k$. Since $b_{2i}(X)\leq  \rank \, H^{2i}(X;\Zp
 2)$ and $b_{\ev}(M^{\Zp {2^l}})=\rank \, H^{\ev}(M^{\Zp {2^l}};\Zp 2)$ by Proposition
 \ref{bettiprop} we see that $b_{2k}(X)=\rank \, H^{2k}(X;\Zp
 2)=1$. Hence, $X$ is orientable.
\end{proof}

\begin{remark}\label{orientrem}
Suppose $F$ is a connected component of the fixed point manifold $M^{\Zp
 n}$, $n>1$. Then one can apply the corollary above for a prime power dividing $n$ to see that $F$ is orientable.
\end{remark}

Next we recall the classical results for the Euler
characteristic and signature of $S^1$-manifolds. By the Lefschetz-fixed point formula
for the Euler characteristic
\begin{equation}\label{eulereq}\chi (M)=\chi
(M^{S^1})=\sum _{Z\subset M^{S^1}} \chi (Z),\end{equation} where the
sum runs over the connected components $Z$ of $M^{S^1}$. An analogous formula holds for the signature if one chooses orientations correctly. Recall that the $S^1$-action induces a complex structure and an orientation on the
normal bundle $\nu _Z$ of $Z$. For later reference we remark that
with respect to this complex structure on $\nu _Z$ the
normal $S^1$-weights at $Z$ are all positive. We
equip $Z$ with the orientation which is compatible with the
orientations of $\nu _Z$ and $M$. It
follows from the rigidity of the equivariant signature (see (A.10)) that
$$sign(M)=sign(M^{S^1})=\sum _{Z\subset
M^{S^1}} sign(Z).$$

Note that if we replace the $S^1$-action by the inverse action (by
composing the action with the isomorphism $S^1\to S^1$, $\lambda
\mapsto \lambda ^{-1}$), the orientation of $Z$ will change if and
only if $Z$ has codimension $\equiv 2 \bmod 4$, i.e. if the dimension of $Z$ is $0$ or $4$. In order to simplify the
discussion we will make the following

\begin{convention}\label{convention} If the $S^1$-action has at least one isolated fixed
point, then we choose one of them, denoted by $pt$, and replace the
action by the inverse action, if necessary, so that $pt$ has
positive orientation.
\end{convention}

Next consider the complex line bundle $\gamma $ over $M$ with $c_1(\gamma )=x$.
Since $b_1(M)=0$ we may lift the action to $\gamma $ (cf. \cite{HaYo76},
Corollary 1.3). We first consider a fixed lift of the $S^1$-action. Let
$Z$ be a connected component of $M^{S^1}$. At a point in $Z$ the
fiber of $\gamma$ is a complex one-dimensional $S^1$-representation whose
isomorphism type is constant on $Z$. We denote the weight of this
representation by $a_Z\in \Z $. For any $l\in \Z$ we can choose a
lift of the $S^1$-action to $\gamma $ such that the $S^1$-equivariant
first Chern class at the connected components is given by
$\{x_{\vert Z}+(a_Z+l)\cdot z\, \mid \, Z\subset M^{S^1}\}$ (see Section \ref{localization section}).

Theorem \ref{maintechnicaltheorem} will follow from a case by
case study of the possible $S^1$-fixed point configurations. Since
$\lbrack x^3\rbrack _M\neq 0$ one gets from the localization
formula (\ref{loc formula}) that $M^{S^1}$ is not empty (see also A.1 in the appendix). Since the Euler
characteristic of $M$ is $<4$ the case of isolated $S^1$-fixed
points cannot occur (see Proposition \ref{bettiprop} and equation
(\ref{eulereq})). Recall that any connected component of $M^{S^1}$ is an
oriented submanifold of even codimension. By Proposition \ref{structure prop} and Proposition \ref{bettiprop} the rational cohomology of each component is in even degrees isomorphic to the one of a complex projective space. Thus, we are left with the following three cases:
\begin{itemize}
\item $M^{S^1}$ is the disjoint union of a connected $4$-dimensional manifold $N$, with $H^{\ev}(N;\Q )\cong H^{\ev}(\C P^2;\Q )$, and a point $pt$.
\item $M^{S^1}$ is the disjoint union of two connected surfaces,  $M^{S^1}=X\cup Y$.
\item $M^{S^1}$ is the disjoint union of a connected surface $X$ and two points $pt$ and $q$.
\end{itemize}

We will show in the next section that none of these cases can occur.

\section{Proof of Theorem \ref{maintechnicaltheorem}}\label{sectionmaintechnicaltheorem}

\subsection{Four-dimensional fixed point components}

The next lemma will be used in the proof.
\begin{lemma}\label{4dimlemma} Let $F\subset M$ be an oriented submanifold of
codimension $2$ and $\gammaneu \cdot x\in H^2(M;\Z)$ its Poincar\'e-dual. Then:
\begin{enumerate}
\item $p_1(F)=(\rho -\gammaneu ^2)\cdot (x_{\vert F})^2$.
\item If $sign (F)=0$, then $(x_{\vert F})^2=0$, $\gammaneu =0$ and the Euler class of the normal bundle of $F\hookrightarrow M$ vanishes.
\end{enumerate}
\end{lemma}

\begin{proof} Consider the normal bundle $\nu $ of $F$ in $M$ equipped with the
orientation compatible with the orientations of $F$ and $M$. Then the Euler class
of $\nu$ is equal to $\gammaneu \cdot (x_{\vert F})$ (where $x_{\vert F}$ denotes the restriction of $x$ to $F$) and
$$\rho \cdot (x_{\vert F})^2=p_1(M)_{\vert F}=p_1(F)+\gammaneu ^2\cdot (x_{\vert F})^2.$$
This shows the first statement.

For the second statement,
note that $sign(F)=0$ implies
$$(\rho -\gammaneu ^2)\cdot (x_{\vert F})^2=p_1(F)=0.$$
Also
$$[(x_{\vert F})^2]_F=[(\gammaneu \cdot x) \cdot x^2]_M=t\cdot \gammaneu .$$
Since $\rho \leq 0$ and $t\neq 0$ these two identities imply $\gammaneu
=0$ and $(x_{\vert F})^2=0$.
\end{proof}

\bigskip
\noindent
We now discuss the case involving a $4$-dimensional connected $S^1$-fixed point component, i.e. $M^{S^1}$ is the disjoint union of a connected
$4$-dimensional manifold $N$, with $H^{\ev}(N;\Q )\cong H^{\ev}(\C
P^2;\Q )$, and a point $pt$. By Convention \ref{convention} $pt$ has positive
orientation. Since $0=sign(M)=sign(M^{S^1})=sign(pt) +sign(N)$ the
signature of $N$ is $-1$ (see (A.10)). Hence, $[x_{\vert N}^2]_N\leq
0$ and $[p_1(N)]_N=-3$ by the signature theorem. From Lemma \ref{4dimlemma} (1) one now obtains the contradiction
$-3=(\rho -\gammaneu ^2)\cdot [x_{\vert N}^2]_N\geq 0$. Hence, $M$ does not support a smooth $S^1$-action with a $4$-dimensional fixed point component.

\subsection{Two $2$-dimensional fixed point components}

In this subsection we discuss the case that $M^{S^1}$ is the disjoint
union of two connected surfaces,  $M^{S^1}=X\cup Y$. As before let
$\gamma $ be the complex line bundle over $M$ with $c_1(\gamma)=x$. We fix a
lift of the $S^1$-action to $\gamma $ such that $S^1$ acts trivially on
the fibers of $\gamma $ over $Y$ (i.e. $a_Y$ vanishes). Since any other
lift differs by a global weight the $S^1$-weights at the connected
components $X$ and $Y$ for a general lift are of the form $a_X+l$
and $l$, respectively, where $l\in \Z $ depends on the choice of the
lift.

Let $x_{Z,1}$ (resp. $y_{Z,i} + n_{Z,i}\cdot z$) denote the
tangential root (resp. normal roots) at a component $Z\subset
M^{S^1}$. We note that by Proposition \ref{structure prop} $[x_{\vert X}]_X\neq 0$ and $[x_{\vert Y}]_Y\neq 0$.

To prove the non-existence of an $S^1$-action with $M^{S^1}=X\cup Y$
we first assume that the $S^1$-action is semi-free around $X$ and
$Y$, i.e. we assume $n_{X,1}=n_{X,2}=n_{Y,1}=n_{Y,2}=1$. In this
case the $S^1$-action on the normal bundles of $X$ and $Y$ coincides
with complex multiplication by $S^1\subset \C $ and the normal
bundles of $X$ and $Y$ each split off a trivial complex line bundle
on dimensional grounds. Hence, we may assume $y_{X,2}=y_{Y,2}=0$,
i.e. the normal weights at $X$ (resp. $Y$) are $\{y_{X,1} + z, z \}$
(resp. $\{y_{Y,1} + z, z \}$).

We first compute $[x^3]_M$ locally. By formula (\ref{eq:13})
$$t=[x^3]_M=-(a_X+l)^3\cdot [y_{X,1}]_X+3(a_X+l)^2\cdot
[x_{\vert X}]_X-l^3\cdot [y_{Y,1}]_Y+3\cdot l^2\cdot
[x_{\vert Y}]_Y .$$
The left hand side is constant in $l$ which gives the relations
$$[y_{Y,1}]_Y=-[y_{X,1}]_X,\quad a_X [y_{X,1}]_X=[x_{\vert X}]_X + [x_{\vert Y}]_Y,\quad a_X([x_{\vert X}]_X - [x_{\vert Y}]_Y)=0$$
$$\text{and } t=[x^3]_M=a_X^2(-a_X\cdot [y_{X,1}]_X+3\cdot
[x_{\vert X}]_X).$$
Next we compute $[p_1(M)\cdot x]_M$ locally. By formula (\ref{eq:17})
$$\rho \cdot t=[p_1(M)\cdot x]_M=2([x_{\vert X}]_X + [x_{\vert Y}]_Y).$$
This leads to
$$a_X(\frac {\rho \cdot t}2-2 \cdot [x_{\vert X}]_X)=-a_X([x_{\vert X}]_X - [x_{\vert Y}]_Y)=0\text{ and }$$
$$t=a_X^2(-a_X\cdot [y_{X,1}]_X+3\cdot
[x_{\vert X}]_X)=a_X^2(-\frac {\rho \cdot t}2+3[x_{\vert X}]_X ).$$
Hence,
$$a_X\neq 0,\quad \rho \cdot t/2=2 \cdot [x_{\vert X}]_X \text{ and }$$
$$t=a_X^2(-\frac {\rho \cdot t}2 + \frac 3 2 \cdot \frac {\rho \cdot t}2)=a_X^2(\frac 1 4\cdot \rho \cdot t).$$
Since $t>0$ and $\rho \leq 0$ the last equation gives a contradiction. Hence, the action on $M$ cannot be semi-free around $X$ and
$Y$.

Next assume the action is not semi-free around $X$ and $Y$. Then
there exists a prime $p$ and a $4$-dimensional connected component
$F\subset M^{\Zp p}$ which contains one of the connected components
of $M^{S^1}$, say $X$. Note that $b_2(F)\geq 1$ and $\rank \,
H^{2}(F;\Zp p)\geq 1$ since $[x_{\vert X}]_X\neq 0$.

\begin{lemma} $X\cup Y\subset F$
\end{lemma}

\begin{proof}
Suppose $Y$ is contained in a connected component $\widetilde F$ of
$M^{\Zp p}$ which is different from $F$. Then $\rank \,
H^{\ev}(M^{\Zp p};\Zp p)\geq \rank \, H^{\ev}(F;\Zp p) + \rank \,
H^{\ev}(\widetilde F;\Zp p)\geq 5$ which contradicts Proposition
\ref{bettiprop}. Hence, $X\cup Y\subset F$.
\end{proof}

By Corollary \ref{orientcor} $F$ is orientable. We fix an
orientation for $F$. Note that the signature of $F$ vanishes since
$sign(F)=\pm sign(X)\pm sign (Y)=0$ (see (\ref{eq:22})). By Lemma
\ref{4dimlemma} the Euler class of the normal bundle of $F$
vanishes. From the spectral sequence for $\pi :F_{S^1}\to BS^1$ we
conclude that the equivariant Euler class is in the image of
$\pi^*$. This implies that the $S^1$-weights of the normal bundle at
$X$ and $Y$ coincide. Since the normal bundle of $F\subset M$
restricted to $X$ (resp. $Y$) is a summand of the normal bundle of
$X\subset M$ (resp. $Y\subset M$) we may assume that the normal
roots of $X$ (resp. $Y$) are given by $\{y_{X,1}+n_{X,1}\cdot
z,n_{X,2}\cdot z\}$ (resp. $\{y_{Y,1}+n_{Y,1}\cdot z,n_{X,2}\cdot
z\}$), i.e. $y_{X,2}=y_{Y,2}=0$ and $n_{X,2}=n_{Y,2}$.

To show that a smooth non-trivial $S^1$-action does not exist we
will first compute the $S^1$-equivariant signature with the
Lefschetz fixed point formula of Atiyah-Bott-Segal-Singer (see A.2 in the
appendix for details).

By the rigidity of the signature the $S^1$-equivariant signature of
$M$ is zero. Using the Lefschetz fixed point formula we get (see
formula (\ref{eq:21})):
$$0=\frac {1+\lambda ^{n_{X,2}}}{1-\lambda ^{n_{X,2}}}\cdot \frac {\lambda ^{n_{X,1}}}{(1-\lambda ^{n_{X,1}})^2}   \cdot [y_{X,1}]_X + \frac {1+\lambda ^{n_{X,2}}}{1-\lambda ^{n_{X,2}}}\cdot \frac {\lambda ^{n_{Y,1}}}{(1-\lambda ^{n_{Y,1}})^2}   \cdot [y_{Y,1}]_Y.$$
By expanding the right hand side around $\lambda =0$ one sees that
either $[y_{X,1}]_X=0=[y_{Y,1}]_Y$ or $[y_{Y,1}]_Y=-[y_{X,1}]_X\neq
0$ and $n_{X,1}=n_{Y,1}$.

If $[y_{X,1}]_X=0=[y_{Y,1}]_Y$, then an inspection of the
localization of $[x^3]_M$ using formula (\ref{eq:13}) gives the
contradiction $0\neq t=[x^3]_M=0$.

If $[y_{Y,1}]_Y=-[y_{X,1}]_X\neq 0$, formula (\ref{eq:13}) gives the relations
$$[x_{\vert X}]_X=[x_{\vert Y}]_Y,\quad \frac {a_X}{n_{X,1}}\cdot [y_{X,1}]_X=2\cdot [x_{\vert X}]_X\text{ and } t=\frac {a_X^2}{n_{X,1}\cdot n_{X,2}}\cdot [x_{\vert X}]_X.$$
In particular, $[x_{\vert X}]_X$ is positive.

Assuming these relations the localization formula (\ref{eq:17}) for $[p_1(M)\cdot x]_M$ leads to
$$\rho \cdot t=4\cdot \frac {n_{X,1}\cdot [x_{\vert X}]_X}{n_{X,2}}.$$
Now $t>0$, $\rho\leq 0$, $[x_{\vert X}]_X>0$, $n_{X,i}>0$ gives the desired contradiction.
Hence, $M$ does not support a smooth $S^1$-action with $M^{S^1}=X\cup Y$.

\subsection{One $2$-dimensional fixed point component and two isolated fixed points}
In this subsection we discuss the remaining case that $M^{S^1}$ is the
disjoint union of a connected surface $X$ and two points $pt$ and
$q$. After dividing out the ineffective kernel we may assume that the $S^1$-action is effective.

By our convention $pt$ has positive orientation. Since the
signature of $M$ is equal to the sum of the signatures of the
$S^1$-fixed point components (see (\ref{eq:22})) the orientation
$\epsilon _q$ of the fixed point $q$ is $-1$.

Let \(n_{X,1},n_{X,2}>0\) be the local weights at \(X\),
\(n_{pt,1},n_{pt,2},n_{pt,3}>0\) the local weights at \(pt\) and
\(n_{q,1},n_{q,2},n_{q,3}>0\) the local weights at \(q\). Since
\(S^1\) acts effectively on \(M\), we have
\(\gcd(n_{X,1},n_{X,2})=\gcd(n_{pt,1},n_{pt,2},n_{pt,3})=\gcd(n_{q,1},n_{q,2},n_{q,3})=1\).

We fix a lift of the \(S^1\)-action on \(M\) into the line bundle
\(\gamma \) with \(c_1(\gamma )=x\) such that the weight \(a_X\) of the
\(S^1\)-representation on the fibers of \(\gamma \) over \(X\) is zero. Recall from Proposition \ref{structure prop} that the restriction of $x$ to $X$ is non-zero.

\begin{lemma}
\label{sec:case-where-ms1=f2-1}
  Let \(F\) be the component of \(M^{\Zp {n_{X,1}}}\) which contains \(X\).
  Then \(F\) contains both isolated fixed points. Moreover, \(F\) is orientable.
\end{lemma}
\begin{proof} If $n_{X,1}=1$ the statement is trivially true. So let $n_{X,1}\geq 2$ and let \(p\) be a prime divisor of \(n_{X,1}\).
  Because \(\gcd(n_{X,1},n_{X,2})=1\),  \(F\) is the component of \(M^{\Zp p}\) which contains \(X\).
Moreover, \(F\) has dimension four. By Corollary \ref{orientcor}
\(F\) is orientable. By Proposition~\ref{bettiprop}, we have \(\rank \, H^{\ev}(M^{\Zp p};\Zp p)=4\).
By Proposition \ref{structure prop} and Poincar\'e duality, we have \(\rank \, H^{\ev}(F;\Zp p)\geq 3\).
Therefore, if \(M^{\Zp p}\) is disconnected, it is the union of \(F\) and a component \(F'\) with \(b_{\ev}(F')= \rank H^{\ev}(F';\Zp p)=1\).

Assume that \(M^{\Zp p}\) is disconnected. It follows from the above
discussion that \(sign\, (F) =\pm 1\). This implies \([p_1(F)]_F=\pm
3\) and \([x_{|F}^2]_F=\pm \alpha\) with \(\alpha\geq 0\). By Lemma
\ref{4dimlemma}, we get the contradiction
$$3=\pm[p_1(F)]_F=\pm(\rho-\gammaneu ^2)[x_{|F}^2]_F=(\rho-\gammaneu ^2)\alpha\leq
0.$$ Therefore \(F=M^{\Zp p}\) is connected.
\end{proof}

\begin{lemma}
\label{sec:case-where-ms1=f2-2}
  Let \(F\) be the connected fixed point component of the $\Zp {n_{pt,1}}$-action on $M$ which contains \(pt\).
  Then \(F\) also contains \(q\) and is orientable.
\end{lemma}
\begin{proof}
If \(F\) is orientable then it contains \(X\) or \(q\), because there is no
orientable manifold which admits an \(S^1\)-action with exactly one
fixed point (see A.2 in the appendix).
In the first case \(n_{pt,1}\) divides a local weight at \(X\), say \(n_{X,1}\).
Then \(F\) contains also the component of \(M^{\Zp {n_{X,1}}}\) which contains \(X\).
Therefore it follows from Lemma~\ref{sec:case-where-ms1=f2-1} that \(F\) also contains \(q\).

In the following we will show that \(F\) is always orientable.
Let   \(k=\gcd(n_{pt,1},n_{pt,2})\) and \(k'=\gcd(n_{pt,1},n_{pt,3})\).
Note that
\begin{equation*}
  \gcd(k,n_{pt,3})=\gcd(k',n_{pt,2})=\gcd(n_{pt,1},n_{pt,2},n_{pt,3})=1.
\end{equation*}
Therefore there are
\(c_1,c_2,c_3\in \mathbb{Z}\) such that
\begin{align*}
n_{pt,1}&=c_1kk', &n_{pt,2}&=c_2k, &n_{pt,3}&=c_3k'.
\end{align*}

If \(n_{pt,1}\) is odd or \(c_1>2\), then the normal bundle of \(F\) admits a complex structure.
Therefore, \(F\) is orientable in this case.

If \(c_1=2\), then there is some \(l>0\) such that \(F\) is a component of \(M^{\Zp {2^l}}\).
Therefore it is orientable by Corollary~\ref{orientcor}.

Hence, we may assume that \(c_1=1\), \(k\) is even and \(k'\) is odd.
Then \(F\) is a component of \((M^{\Zp 2})^{\Zp {k'}}\).
Therefore it is orientable by Corollary~\ref{orientcor}.
\end{proof}

\begin{lemma}
  The normal weights at \(pt\) and \(q\) are equal up to ordering.
\end{lemma}
\begin{proof}
  Up to ordering there are the following three cases:
  \begin{enumerate}
  \item \(n_{pt,1}|n_{pt,2}\), \(n_{pt,3}\nmid n_{pt,2}\), \(\gcd(n_{pt,1},n_{pt,3})=1\), \(n_{pt,1}\neq 1\),
  \item \(n_{pt,1}|n_{pt,2}\), \(n_{pt,3}|n_{pt,2}\), \(\gcd(n_{pt,1},n_{pt,3})=1\), \(n_{pt,1}\neq 1\),
  \item if \(n_{pt,i}\neq 1\), then we have, for \(j\neq i\), \(n_{pt,i}\nmid n_{pt,j}\).
  \end{enumerate}
Before we consider these cases we prove the following two claims.

\emph{Claim 1:} If \(n_{pt,i_1}\nmid n_{pt,i_2}\) and  \(n_{pt,i_1}\nmid n_{pt,i_3}\), then there is exactly one \(j\in \{1,2,3\}\) such that \(n_{pt,i_1}|n_{q,j}\).
Moreover, we have \(n_{pt,i_1}=n_{q,j}\).

By Lemma~\ref{sec:case-where-ms1=f2-2}, the component of \(M^{\Zp {n_{pt,i_1}}}\) which contains \(pt\) also contains \(q\).
Moreover, this component has dimension two.
Therefore \(n_{pt,i_1}\) divides exactly one of the local weights at \(q\), say \(n_{q,j}\).
Again by Lemma~\ref{sec:case-where-ms1=f2-2} the component of \(M^{\Zp {n_{q,j}}}\) which contains \(q\) also contains \(pt\).
Hence, \(n_{q,j}\) divides one of the local weights at \(pt\).
It follows from the assumptions in the claim that this weight must be \(n_{pt,i_1}\).
Therefore \(n_{pt,i_1}=n_{q,j}\) follows. This proves Claim 1.

\emph{Claim 2:} If \(n_{pt,i_1}|n_{pt,i_2}\) and \(n_{pt,i_1}\nmid n_{pt,i_3}\), then there are exactly two \(j_1,j_2\in\{1,2,3\}\) such that \(n_{pt,i_1}|n_{q,j_1}\) and \(n_{pt,i_1}|n_{q,j_2}\).
Moreover, we have \((n_{pt,i_1},n_{pt,i_2})=(n_{q,j_1},n_{q,j_2})\) up to ordering.

By Lemma~\ref{sec:case-where-ms1=f2-2}, the component of \(M^{\Zp {n_{pt,i_1}}}\) which contains \(pt\) also contains \(q\).
Moreover, this component has dimension four.
Therefore \(n_{pt,i_1}\) divides exactly two of the local weights at \(q\), say \(n_{q,j_1}\) and \(n_{q,j_2}\).

At first assume \(n_{pt,i_1}\neq n_{pt,i_2}\).
Then, by Claim 1, applied to \(n_{pt,i_2}\), we know that exactly one of these weights is equal to \(n_{pt,i_2}\).
Denote this weight by \(n_{q,j_2}\).
By Lemma~\ref{sec:case-where-ms1=f2-2} the component of \(M^{\Zp {n_{q,j_1}}}\) which contains \(q\) also contains \(pt\).
Therefore \(n_{q,j_1}\) divides \(n_{pt,i_1}\) or \(n_{pt,i_2}\).
In the second case this component has dimension four.
Hence, \(n_{q,j_1}\) divides also \(n_{pt,i_1}\).
This implies \(n_{pt,i_1}=n_{q,j_1}\).

Now assume that \(n_{pt,i_1}=n_{pt,i_2}\).
Then, by Lemma~\ref{sec:case-where-ms1=f2-2}, the component of  \(M^{\Zp {n_{q,j_1}}}\) which contains \(q\) also contains \(pt\).
Therefore \(n_{q,j_1}\) divides \(n_{pt,i_1}\).
Hence, \(n_{q,j_1}=n_{pt,i_1}\) and by the same argument \(n_{pt,i_1}=n_{q,j_2}\).
This proves the second claim.

Now consider the three cases mentioned above.
In the first case the statement of the lemma follows from Claim 1 applied to \(n_{pt,i_1}=n_{pt,3}\) and Claim 2 applied to \(n_{pt,i_1}=n_{pt,1}\).

In the situation of the second case at first assume that \(n_{pt,3}\neq1\).
Then the statement of the lemma follows from Claim 2 applied to both \(n_{pt,i_1}=n_{pt,1}\) and \(n_{pt,i_1}=n_{pt,3}\).
Now assume that \(n_{pt,3}=1\). Then it follows from Claim 2 applied to \(n_{pt,i_1}=n_{pt,1}\) that there are two weights \(n_{q,1}\) and \(n_{q,2}\) such that \((n_{pt,1},n_{pt,2})=(n_{q,1},n_{q,2})\) up to ordering.
If we assume that \(n_{q,3}\neq 1\), we get from Claim 1 or Claim 2 applied to \(n_{pt,i_1}=n_{q,3}\) that there is a \(n_{pt,j}\) which is equal to \(n_{q,3}\).
This leads to a contradiction because only two of the local weights at \(q\) are divisible by \(n_{pt,1}\).
Therefore the local weights at \(pt\) and \(q\) are the same up to ordering.

In the third case first apply Claim 1 to all \(n_{pt,i}\neq 1\) to show that each of these weights is equal to exactly one local weight at \(q\).
As in the previous case it follows from an application of the Claims 1 and 2 to the local weights at \(q\) that \(\#\{i;\; n_{q,i}=1\}=\#\{i;\; n_{pt,i}=1\}\).
Therefore the lemma follows in this case.
\end{proof}

\begin{lemma}
\label{sec:one-2-dimensional}
  The case \((n_{pt,1},n_{pt,2},n_{pt,3})=(n_{q,1},n_{q,2},n_{q,3})\) does not occur.
\end{lemma}
\begin{proof}
  Assume that we are in this case.
Since the lift of the \(S^1\)-action into \(\gamma \) is not unique we get
from the localization formulas for \([x^3]_M\) and \([p_1(M)\cdot
x]_M\) two polynomials in a variable \(l\) which are equal to
\([x^3]_M\) and \([p_1(M)\cdot x]_M\), respectively. By comparing
coefficients we get the following equations (see appendix):
  \begin{align}
\label{eq:4}
    0&=\frac{3}{n_{X,1}n_{X,2}}[x]_X + \frac{3a_{pt}}{n_{pt,1}n_{pt,2}n_{pt,3}} - \frac{3a_q}{n_{pt,1}n_{pt,2}n_{pt,3}},\\
\label{eq:5}
    0&= \frac{3a_{pt}^2}{n_{pt,1}n_{pt,2}n_{pt,3}} - \frac{3a_q^2}{n_{pt,1}n_{pt,2}n_{pt,3}},\\
\label{eq:6}
    [x^3]_M&= \frac{a_{pt}^3}{n_{pt,1}n_{pt,2}n_{pt,3}} - \frac{a_q^3}{n_{pt,1}n_{pt,2}n_{pt,3}} \quad \text{ and }\\
\label{eq:7}
    [p_1(M)\cdot x]_M&= \frac{n_{X,1}^2+n_{X,2}^2}{n_{X,1}n_{X,2}}[x]_X + (a_{pt}-a_q)\frac{n_{pt,1}^2+n_{pt,2}^2+n_{pt,3}^2}{n_{pt,1}n_{pt,2}n_{pt,3}}.
  \end{align}

Because of (\ref{eq:5}) we have \(a_{pt}=\pm a_q\).
Then (\ref{eq:6}) and \([x^3]_M>0\) implies \(a_{pt}=-a_q>0\).

>From (\ref{eq:4}) and (\ref{eq:7}) we get
\begin{equation*}
  [p_1(M)\cdot x]_M= -2a_{pt}\frac{n_{X,1}^2+n_{X,2}^2}{n_{pt,1}n_{pt,2}n_{pt,3}} + 2a_{pt}\frac{n_{pt,1}^2+n_{pt,2}^2+n_{pt,3}^2}{n_{pt,1}n_{pt,2}n_{pt,3}}.
\end{equation*}

Now, using the divisibility properties implied by
Lemma~\ref{sec:case-where-ms1=f2-1}, it follows that the right hand
side of this equation is always positive. This is a contradiction to
our assumption.
\end{proof}

By combining the above lemmas we see that there is no \(S^1\)-action
on \(M\) with a fixed point set consisting out of a two-dimensional
component and two isolated fixed points.

\section{Complete intersections with $T^2$-action}\label{finiteness section}
\label{sec:t2}

In this section we will extend the vanishing results for indices of twisted $Spin^c$-Dirac operators and twisted elliptic genera given in \cite{De99}, \cite[Section 4]{DeTop} to prove that there are only finitely many complete intersections which admit an effective action of a two-dimensional torus \(T^2\) in each odd complex dimension. As a corollary we exhibit in each odd complex dimension $\geq 3$ complete intersections with a metric of positive Ricci-curvature (in fact with positive first Chern class) but no effective action of a two-dimensional torus.  We also give a proof of the Petrie conjecture for \(T^2\)-actions in complex dimension $<12$.
The main new technical ingredient which is needed to prove these claims is the following lemma.

\begin{lemma}
  \label{sec:torus-acti-stab}
  Let \(T\) be a torus.
  Let \(M\) be a \(T\)-manifold with \(\rank T > b_2(M)\) and \(a\in H^4(M_T;\Q)\) such that the restriction of $a$ to $H^4(M;\Q)$ vanishes.
  Then there is a non-trivial homomorphism \(\rho(S^1,T): S^1\rightarrow T\) such that \(\rho(S^1,T)^*(a)\in \pi_{S^1}^*(H^4(BS^1;\Q))\).
\end{lemma}
Here $\pi_{S^1}: M_{S^1}\to BS^1$ is the projection in the Borel construction.
\begin{proof}
>From the Serre spectral sequence for the fibration \(M\rightarrow M_T\rightarrow BT\) we have the following direct sum decomposition of the \(\Q\)-vector space \(H^4(M_T;\Q)\),
  \begin{equation*}
    H^4(M_T;\Q)\cong E^{0,4}_\infty \oplus E^{2,2}_\infty \oplus E^{4,0}_\infty.
  \end{equation*}
Moreover, we have
\begin{align*}
  E^{0,4}_\infty &\subset H^{4}(M;\Q),& E^{2,2}_\infty &\subset E^{2,2}_2/d_2(E^{0,3}_2),& E^{4,0}_\infty&=\pi_{S^1}^*H^4(BT;\mathbb{Q}).
\end{align*}
Let \(a_{0,4}\), \(a_{2,2}\), \(a_{4,0}\) be the components of \(a\) according to this decomposition.
Then \(a_{0,4}=0\) by assumption.
Moreover, there is an \(\tilde{a}_{2,2}\in E_2^{2,2}\) such that \(a_{2,2}=[\tilde{a}_{2,2}]\).

Now it is sufficient to find a non-trivial homomorphism \(\rho(S^1,T): S^1\rightarrow T\) such that \(\rho(S^1,T)^*(\tilde{a}_{2,2})=0\).
We have the following isomorphisms:
\begin{align*}
  E_2^{2,2}&\cong H^2(BT;\mathbb{Q})\otimes H^2(M;\mathbb{Q})\\
  &\cong \left(H^2(BT;\Q)\right)^{b_2(M)}.
\end{align*}
Since \(\rank T > b_2(M)\) we can find a non-trivial homomorphism \(\phi: H^2(BT;\Q)\rightarrow H^2(BS^1;\Q)=\Q\) such that all components of \(\tilde{a}_{2,2}\) according to the above decomposition of \(E_2^{2,2}\) are mapped to zero by \(\phi\).
After scaling, we may assume that \(\phi\) is induced by a surjective homomorphism \(H^2(BT;\mathbb{Z})\rightarrow H^2(BS^1;\mathbb{Z})\).
By dualizing we get a homomorphism \(\hat{\phi}: H_2(BS^1;\mathbb{Z})\rightarrow H_2(BT;\mathbb{Z})\).
Since for any torus \(H_2(BT;\mathbb{Z})\) is naturally isomorphic to the integer lattice in the Lie algebra \(LT\) of \(T\), \(\hat{\phi}\) defines the desired homomorphism.
\end{proof}
We shall now combine this lemma with the methods developed in \cite{De99,DeTop} to study $T$-actions on certain manifolds with \(\rank T > b_2(M)\). Since we are mainly interested in application for manifolds which are cohomologically complete intersections we will restrict to the case that $b_2(M)=1$ and $\rank T=2$ (see \cite{Wi15} for other results).

Let $M$ be a $2n$-dimensional
$Spin^c$-manifold with $b_1(M)=0$ and $H^2(M;\Z )\cong \Z$ generated by $x$. Note that this situation applies to any integral cohomology $\C P^n$ and, also, to any integral cohomology complete intersection of complex dimension $n> 2$. Let $\gamma$ be the complex line bundle with $c_1(\gamma )=x$.

Suppose a torus $T^2$ of rank two acts effectively and smoothly on $M$. Then we can lift the $T^2$-action to the $Spin^c$-structure \cite[Lemma 2.1]{Wi13} and to $\gamma$ \cite[Corollary 1.2]{HaYo76}. We fix a lift to the $Spin^c$-structure. Note that for a given $T^2$-fixed point we may choose the lift to $\gamma $ such that the restriction of the line bundle to the point is a trivial $T^2$-representation.

Let us for a moment restrict the $T^2$-action to $S^1$ with respect to a homomorphism $S^1\to T^2$ and consider a connected $S^1$-fixed point component $Z$. The fiber of the normal bundle of $Z\subset M$ at a point $pt\in Z$ is a real $S^1$-representation. We denote its weights by $\pm  m_{Z,j}$. The fiber of $\gamma $ over $pt$ is a complex one dimensional $S^1$-representation. Its weight will be denoted by $a_Z$. Note that the normal weights $\pm  m_{Z,j}$ and $a_Z$ only depend on $Z$ but not on the chosen point $pt$.

Let \(Z_0,\dots, Z_{k_1}\) be the components of \(M^{S^1}\) and  \(n_i=\frac{1}{2}\dim Z_i\). We will choose the lift of the $T^2$-action to $\gamma $ such that the weight vanishes at one of the components, say $Z_0$, i.e.  $a_{Z_0}=0$.

\begin{prop}
\label{sec:compl-inters-with-2} Let $M$ and $\gamma $ be as above. Suppose $H^{\ev}(M;\mathbb{Q})=\mathbb{Q}[x]/(x^{n+1})$. If the equivariant first Pontrjagin class \(p_1(TM\oplus \bigoplus_{i=1}^k \gamma)_{S^1}$ is in $\pi_{S^1}^*(H^4(BS^1;\mathbb{Q}))\) then we have \(k<n\).
\end{prop}
\begin{proof} At first we replace the $S^1$-action by the two-fold action. Recall from Proposition \ref{structure prop} that  \(H^{\ev}(Z_i;\mathbb{Q})\cong H^{\ev}(\C P^{n_i};\mathbb{Q})\) and that \(\sum_i (n_i+1)=n+1\).

Let $V$ be the $S^1$-equivariant complex vector bundle given by
\begin{equation*}
  V= n_0 \gamma\oplus\bigoplus_{i=1}^{k_1} (n_i+1) \gamma\otimes_{\mathbb{C}}W_{-a_i},
\end{equation*}
where \(W_a\) denotes the one-dimensional unitary \(S^1\)-representation of weight \(a\) and \(a_i:=a_{Z_i}\). Let
$${\mathcal U}_V:=\bigotimes_{n=1}^\infty S_{q^n}(\widetilde {TM}\otimes _\R \C )\otimes
\Lambda _{-1}(V^*)\otimes \bigotimes_{n=1}^\infty \Lambda _{-q^n}(\widetilde
{V}\otimes _\R \C ).$$ Here $q$ is a formal variable,
$\widetilde E$ denotes the reduced vector bundle $E-\dim (E)$ and
$\Lambda _t:=\sum \Lambda ^i\cdot t^i$ (resp. $S_t:=\sum S^i\cdot t^i$) denotes the exterior
(resp. symmetric) power operation. We now consider the equivariant $Spin^c$-Dirac operator twisted with
${\mathcal U}_V$. Its index is a $q$-power series of virtual $S^1$-representations and will be denoted by
$$ind_{S^1}(\partial _{c}\otimes {\mathcal U}_V)\in R(S^1)[[ q ]].$$

Using the Atiyah-Singer index theorem one computes that the non-equivariant index $ind(\partial
_{c}\otimes {\mathcal U}_V)$ is equal to $[x^n]_M$ for $q=0$. In particular, $ind(\partial
_{c}\otimes {\mathcal U}_V)$ is non-zero. Using Proposition 3.1 of \cite{DeTop} one can show that the equivariant index $ind_{S^1}(\partial _{c}\otimes {\mathcal U}_V)$ converges for \(q=e^{2\pi i \tau}\), \(\tau\in \mathcal{H}\), and \(\lambda=e^{2\pi i \tilde{z}}\) a topological generator of \(S^1\subset \C\) to a holomorphic function \(f(\tau,\tilde{z})\) on the product \(\mathcal{H}\times \C\) where \(\mathcal{H}\) is the upper half plane (for details on this argument see the proof of Proposition 4.1  in \cite{DeTop}).
Moreover one can show that
\begin{equation*}
  f(\tau,\tilde{z})=e(\tilde{z})F(\tau,\tilde{z}),
\end{equation*}
where \(e(\tilde{z})\) is a holomorphic function and \(F\) is a holomorphic Jacobi function for \(SL_2(\mathbb{Z}) \ltimes \mathbb{Z}^2\) of index
\begin{equation*}
  I=\frac{1}{2}\left(\sum_{i=1}^{k_1}(n_i+1)a_i^2-\sum_{j=1}^n m_{Z_0,j}^2\right).
\end{equation*}
Note that \(I\) is an integer because we are looking at the two-fold action.

Since $ind(\partial
_{c}\otimes  {\mathcal U}_V)\neq 0$, $F$ does not vanish identically.
Because a holomophic Jacobi function with negative index vanishes identically we have
\begin{equation*}
  \sum_{j=1}^n m_{Z_0,j}^2 \leq \sum_{i=1}^{k_1}(n_i+1)a_i^2.
\end{equation*}

Next note that the restriction of $p_1(TM\oplus \bigoplus_{i=1}^k \gamma)_{S^1}$ to a point in $Z_i$ is equal to $(\sum_{j=1}^n m_{Z_i,j}^2 + k a_i^2)\cdot z^2\in H^4(BS^1;\Z)$ and is independent of the choice of $Z_i$ since
$$p_1(TM\oplus \bigoplus_{i=1}^k \gamma)_{S^1}\in \pi_{S^1}^*(H^4(BS^1;\mathbb{Q})) .$$
We may assume that \(a_1^2=\max_i\{a_i^2\}\). Then we get
\begin{equation*}
  \sum_{j=1}^n m_{Z_1,j}^2 + k a_1^2 =\sum_{j=1}^n m_{Z_0,j}^2\leq \sum_{i=1}^k (n_i+1)a_i^2 \leq n a_1^2.
\end{equation*}
Hence the claim follows.
\end{proof}

Let us call a smooth manifold $M$ with $H^*(M;\Z)\cong H^*(X_n(d_1,\dots,d_r);\Z )$ an {\em integral cohomology $X_n(d_1,\dots,d_r)$}. With the ingredients above we can prove the following theorem:

\begin{theorem}
\label{sec:compl-inters-with-3}
  Let $M$ be an integral cohomology $X_n(d_1,\dots,d_r)$ with \(n \geq 3 \) odd, $x$ a generator of $H^2(M;\Z)$ and $p_1(M)=-kx^2$. Suppose $M$ admits an effective action of a two-dimensional torus.
  Then we have \(k <n\).
\end{theorem}
\begin{proof}
  We may assume that  \(k> 0\). We lift the \(T^2\)-action into \(\gamma\) in such a way that the action on the fibers over one of the \(T^2\)-fixed points in \(M\) is trivial.

Note that \(p_1(TM\oplus \bigoplus_{i=1}^k \gamma)=0\).
  By Lemma~\ref{sec:torus-acti-stab}, there is a non-trivial homomorphism \(S^1\to T^2\) such that  \(p_1(TM\oplus \bigoplus_{i=1}^k \gamma)_{S^1}\in \pi_{S^1}^*(H^4(BS^1;\Q))\).
  Therefore the claim follows from Proposition~\ref{sec:compl-inters-with-2}.
\end{proof}

As an immediate corollary of the above theorem and Proposition \ref{compinterprop} we get:

\begin{corollary}\label{finiteness corollary}
  For each odd \(n\), there are, up to diffeomorphism, only finitely many complete intersection of complex dimension \(n\) which admit an effective action of a two-dimensional torus.
\end{corollary}

We also get the following corollary:

\begin{corollary}\label{Ricci Beispiele}
  For each \(m=4k+2\), \(k\geq 1\), there is a manifold \(M\) of dimension \(m\) which admits a metric of positive Ricci-curvature but no effective action of a two-dimensional torus.
\end{corollary}
\begin{proof}
  Let \(M=X_{2k+1}(d_1,\dots,d_r)\) with
$$ \sum_i d_i <(2k+1)+r+1\quad \text{ and }\quad  \sum_id_i^2\geq 2(2k+1)+r+1.$$
Then \(M\) is a K\"ahler manifold with positive first Chern-class.
Therefore \(M\) admits a K\"ahler metric with positive Ricci-curvature \cite{Ya78}.
But by Theorem~\ref{sec:compl-inters-with-3} there is no effective \(T^2\)-action on \(M\).
\end{proof}

In \cite{DeTop} it has been shown that if a homotopy complex projective space \(M\)  of dimension \(2n<24\) admits a non-trivial \(SU(2)\)-action with fixed point, then the Pontrjagin classes of \(M\) are standard.
Using Lemma~\ref{sec:torus-acti-stab} and arguments similar to the above discussion and the proofs in \cite{De99, DeTop} one can also prove the following theorem about homotopy complex projective spaces.

\begin{theorem}\label{Petrie theorem}
  Let \(M\) be homotopy equivalent to \(\C P^n\) with \(n<12\).
  If \(M\) admits an effective action of a two-dimensional torus, then the Pontrjagin classes of \(M\) are standard, i.e. $p(M)=(1+x^2)^{n+1}$ where \(x\in H^2(M;\mathbb{Z})\) is a generator of the cohomology ring of \(M\).
\end{theorem}
\begin{proof} Let \(b\in \mathbb{Z}\) be such that \(p_1(M)=b x^2\) and let \(\gamma\) be the line bundle over \(M\) with \(c_1(M)=x\). Then, by \cite[Corollary 1.3]{HaYo76}, the \(T^2\)-action on \(M\) lifts into \(\gamma\) in such a way that the \(T^2\)-action on the fiber of \(\gamma\) over a fixed point \(pt\in M^{T^2}\) is trivial.

By combining Proposition~\ref{sec:compl-inters-with-2} and Lemma \ref{sec:torus-acti-stab}, we have \(b>-n\). By combining Lemma~\ref{sec:torus-acti-stab} with the proof of Theorem 4.2 of \cite{De99}, we see that \(b\leq n+1\).
  Moreover, in case of equality we have \(p(M)=(1+x^2)^{n+1}\) (cf. {\it loc. cit.}).
  Since \(p_1(M)\mod 24\) is determined by the homotopy type of \(M\), it follows that for \(n<12\) the Pontrjagin classes of \(M\) are standard.
\end{proof}

\begin{appendix}

\section{Localization formulas for equivariant cohomology classes and equivariant signatures}

In the appendix we provide the local formulas for equivariant
cohomology classes and equivariant signatures which are used
in the proof.

Let $M$ be a closed smooth oriented $m$-dimensional manifold with smooth $S^1$-action. Let $Z$ be a connected component of the fixed point manifold $M^{S^1}$. We denote the tangential
formal roots of $Z$ by $\pm x_{Z,j}$. Hence, the total Pontrjagin class of $Z$ is given by $p(Z)=\prod _j(1+x_{Z,j}^2)$.

Let $\nu _Z$ be the normal bundle of $Z$. We orient $\nu _Z$ via the complex structure induced
by the $S^1$-action and fix the orientation for $Z$ which is compatible to the
orientation of $M$ and $\nu _Z$.

We denote the $S^1$-equivariant Euler class of the normal bundle
$\nu _Z$ by $e_{S^1}(\nu _Z)$. The normal bundle $\nu _Z$ decomposes
as a direct sum of complex vector bundles corresponding to the
$S^1$-representations. Applying the splitting principle to these we
can associate to $\nu _Z$ $S^1$-equivariant roots $y_{Z,i} +
n_{Z,i}\cdot z$, where the $y_{Z,i}$'s are non-equivariant formal
roots of the corresponding bundle, the weights $n_{Z,i}\in \Z $ are
positive by our convention, and $z$ is a formal variable which one
should think of as a fixed generator of the integral lattice of
$S^1$ or a fixed generator of $H^2(BS^1;\Z )$. With this notation
the $S^1$-equivariant Euler class of the normal bundle $\nu _Z$ is
given by $e_{S^1}(\nu _Z)=\prod _i(y_{Z,i} + n_{Z,i}\cdot z)$.

\subsection{Equivariant cohomology classes} Let $[\quad ]_M$ denote the push forwards $H^*(M_{S^1};\Z )\to H^{*-n}(pt_{S^1};\Z )=H^{*-n}(B{S^1};\Z )$ and $H^*(M;\Z )\to H^{*-n}(pt;\Z )$. Note that the first is also called {\em integration over the fiber} and the latter can also be described by evaluation on the
fundamental cycle of $M$. Let $\bar y\in H^*(M_{S^1};\Z )$ be an equivariant cohomology class and $y$ its image under the restriction map $H^*(M_{S^1};\Z )\to H^*(M;\Z )$. Then, by naturality, $[y]_M$ is equal to the zero degree part of $[\bar y]_M$. In particular, if $y$ is homogeneous of degree $m$ then $[y]_M$ can be computed by integrating its equivariant lift $\bar y$ over the fiber, i.e. $[y]_M=[\bar y]_M$. If $\bar y$ is homogeneous of degree $>m$ then $[y]_M$ vanishes for dimensional reasons but $[\bar y]_M$ may be non-trivial. This may also lead to interesting applications. However, we will only need the case $\deg y=m$.

By the localization formula (\ref{loc formula}) in equivariant cohomology of Atiyah-Bott and Berline-Vergne \cite{BeVe82, AtBo84} the class $[\bar y]_M$ can be computed in terms of local data at $M^{S^1}$:
$$ [y]_M=[\bar y]_M=\sum _{Z\subset M^{S^1}} \mu
(\bar y,Z).$$
We will now apply the localization formula for the $6$-dimensional manifold $M$ considered in Section \ref{section prelim} and to the classes $x^3$ and $p_1(M)\cdot x$. Note that $x$ can be lifted to an equivariant class since $b_1(M)=0$ and $p_1(M)$ lifts canonically to the first Pontrjagin class of the equivariant tangent bundle $TM_{S^1}\to M_{S^1}$.

Fix a lift of the $S^1$-action to the
complex line bundle $\gamma $ with $c_1(\gamma )=x$. Then at a connected
component $Z\subset M^{S^1}$ the $S^1$-equivariant first Chern class
of $\gamma $ has the form $x_{\vert Z}+a_Z\cdot z$.

The lift is not unique. For any $l\in \Z$ we can choose a lift of
the $S^1$-action such that the $S^1$-equivariant first Chern class
at the connected components is given by $\{x_{\vert Z}+(a_Z+l)\cdot
z\, \mid \, Z\subset M^{S^1}\}$.

The localization formula for $x^3$ with
respect to such a lift takes the form
  $$[x^3]_M=\sum _{Z\subset M^{S^1}} \mu
(x^3,Z),$$
where the local datum $\mu
(x^3,Z)$ at $Z$ is given by
$$\mu
(x^3,Z)=[(x_{\vert Z}+(a_Z+l)\cdot z)^3\cdot e_{S^1}(\nu _Z)^{-1}]_Z.$$ Note that the sum $\sum _{Z\subset M^{S^1}} \mu
(x^3,Z)$ is independent of the
parameter $l$.

\bigskip
\noindent Depending on the dimension of $Z$ the local datum $\mu
(x^3,Z)$ for $x^3$ at $Z$ takes the form:

\bigskip
\noindent If $Z$ is a point, then \begin{equation}\label{eq:15} \mu
(x^3,Z)=\epsilon_Z\cdot \frac {((a_Z+l)\cdot z)^3}{ n_{Z,1}\cdot
n_{Z,2}\cdot n_{Z,3}\cdot z^3}= \epsilon_Z\cdot \frac
{(a_Z+l)^3}{n_{Z,1}\cdot n_{Z,2}\cdot n_{Z,3}},\end{equation}
 where
$\epsilon_Z\in \{\pm 1\}$ is $+1$ if and only if the point $Z$ is positively
oriented.

\bigskip
\noindent If $Z$ is $2$-dimensional, then $$\mu (x^3,Z)=\left
[(x_{\vert Z}+(a_Z+l)\cdot z)^3\cdot \left((y_{Z,1} + n_{Z,1}\cdot
z)\cdot (y_{Z,2} + n_{Z,2}\cdot z)\right )^{-1}\right ]_Z=$$
\begin{equation}\label{eq:13} \frac 1 {n_{Z,1}\cdot
n_{Z,2}}\cdot \left (-\frac {(a_Z+l)^3}{n_{Z,1}}\cdot [y_{Z,1} ]_Z-
\frac {(a_Z+l)^3}{n_{Z,2}}\cdot [y_{Z,2}]_Z+3(a_Z+l)^2\cdot
[x_{\vert Z}]_Z\right ).\end{equation}

\bigskip
\noindent If $Z$ is $4$-dimensional, then
$$\mu (x^3,Z)=\left [(x_{\vert Z}+(a_Z+l)\cdot z)^3\cdot (y_{Z,1} + n_{Z,1}\cdot z)^{-1}\right ]_Z=$$
\begin{equation}\label{eq:14}\frac{3(a_Z+l)}{n_{Z,1}}\cdot [x_{\vert Z}^2]_Z-\frac
{3(a_Z+l)^2}{n_{Z,1}^2}[x_{\vert Z}\cdot y_{Z,1}]_Z+\frac
{(a_Z+l)^3}{n_{Z,1}^3}[y_{Z,1}^2]_Z.\end{equation}

Note that, if $a_Z+l=0$, then the local data in (\ref{eq:15}), (\ref{eq:13}), (\ref{eq:14})
vanish. Note also that the local datum in
(\ref{eq:14}) vanishes for any $l$ if
$b_2(Z)=0$.

\bigskip
\noindent Next we provide formulas for $p_1(M)\cdot x$. Note that the Pontrjagin classes lift canonically to $S^1$-equivariant Pontrjagin classes. Thus, we may compute $[p_1(M)\cdot x]_M$ using the localization formula.

\bigskip
\noindent Depending on the dimension of $Z$ the local datum $\mu
(p_1(M)\cdot x,Z)$ for $p_1(M)\cdot x$ at $Z$ takes the form:

\bigskip
\noindent If $Z$ is a point, then the local datum is equal to
$\epsilon_Z\cdot \frac{(a_Z+l)\cdot z\cdot
(n_{Z,1}^2+n_{Z,2}^2+n_{Z,3}^2)\cdot z^2}{ n_{Z,1}\cdot n_{Z,2}\cdot
n_{Z,3}\cdot z^3}$, i.e.
\begin{equation}\label{eq:16}\mu (p_1(M)\cdot
x,Z)=\epsilon_Z\cdot \frac {(a_Z+l)\cdot
(n_{Z,1}^2+n_{Z,2}^2+n_{Z,3}^2)}{n_{Z,1}\cdot n_{Z,2}\cdot
n_{Z,3}},\end{equation} where $\epsilon_Z\in \{\pm 1\}$ is $+1$ if and only if
$Z$ is positively oriented.

\bigskip
\noindent If $Z$ is $2$-dimensional, then the local datum is equal to
\begin{multline*}
\left [(x_{Z,1}^2+(y_{Z,1} + n_{Z,1} z)^2+(y_{Z,2} + n_{Z,2}z)^2)\cdot (x_{\vert Z}+(a_Z+l)z)\cdot\right.\\
 \left.\cdot ((y_{Z,1} + n_{Z,1}z)(y_{Z,2} + n_{Z,2} z) )^{-1}\right ]_Z
\end{multline*}
which gives
$$\mu (p_1(M)\cdot
x,Z)=-\frac {(a_Z+l)}{n_{Z,1}\cdot n_{Z,2}}(n_{Z,1}^2 + n_{Z,2}^2)(\frac 1 {n_{Z,1}}\cdot [y_{Z,1}]_Z + \frac 1 {n_{Z,2}}\cdot [y_{Z,2}]_Z)$$
\begin{equation}\label{eq:17}+ \frac {n_{Z,1}^2 + n_{Z,2}^2}{n_{Z,1}\cdot n_{Z,2}}\cdot [x_{\vert
Z}]_Z+2\frac {(a_Z+l)}{n_{Z,1}\cdot n_{Z,2}}(n_{Z,1}\cdot [y_{Z,1}]_Z + n_{Z,2}\cdot [y_{Z,2}]_Z).\end{equation}

\bigskip
\noindent If $Z$ is $4$-dimensional, then $\mu (p_1(M)\cdot
x,Z)$ is equal to
$$\left [(p_1(Z)+(y_{Z,1} + n_{Z,1}\cdot z)^2)
\cdot(x_{\vert Z}+(a_Z+l)\cdot z)\cdot (y_{Z,1} + n_{Z,1}\cdot
z)^{-1}\right ]_Z$$
\begin{equation}\label{eq:18}=[x_{\vert Z}\cdot
y_{Z,1}]_Z+\frac{a_Z+l}{n_{Z,1}}\cdot [p_1(Z)]_Z.\end{equation}

Note that, if $a_Z+l=0$, then the local datum at $Z$ vanishes if $Z$ is a point. Note also that the local datum at $Z$ vanishes for any $l$ if the
dimension of $Z$ is $4$ and $b_2(Z)=0$.

\subsection{Equivariant signatures}

In this section we recall the Lefschetz fixed point formula for the equivariant signature and provide some formulas used in the paper.

Let $M$ be an oriented closed manifold with smooth $S^1$-action and
let $sign_{S^1}(M)$ denote the $S^1$-equivariant signature. A priori
$sign_{S^1}(M)$ is an element of the representation ring $R(S^1)$
which we identify via the character with the ring of finite Laurent
polynomials $\Z[\lambda ,\lambda ^{-1}]$.

By the Lefschetz fixed point formula of Atiyah-Bott-Segal-Singer
(cf. \cite{AtSi68}) the $S^1$-equivariant signature can be computed
locally at the $S^1$-fixed point components. More precisely, for any
topological generator $\lambda \in S^1$
\begin{equation}\label{eq:19}sign_{S^1}(\lambda )=\sum _{Z\subset M^{S^1}} \mu _Z(\lambda ),\end{equation}
where the local datum $\mu _Z(\lambda )$ at a connected component $Z\subset M^{S^1}$ is given by
$$\mu _Z(\lambda ) =\left [ \prod _j x_{Z,j}\cdot \frac {1+e^{-x_{Z,j}}}
{1-e^{-x_{Z,j}}}\cdot \prod _i \frac {1+\lambda^{-n_{Z,i}}\cdot e^{-y_{Z,i}}}{1-\lambda^{-n_{Z,i}}\cdot e^{-y_{Z,i}}}\right ]_Z.$$
For example, if $M$ is $6$-dimensional and $Z$ is a point, then the local datum is given by
\begin{equation}\label{eq:20}\mu _Z(\lambda )=\epsilon _Z\cdot \prod _{i=1}^3 \frac {1+\lambda^{-n_{Z,i}}}{1-\lambda^{-n_{Z,i}}}.\end{equation}

If $Z$ is $2$-dimensional, then the local datum is given by
\begin{equation}\label{eq:21}\mu _Z(\lambda )=\left [ x_{Z,1}\cdot \frac {1+e^{-x_{Z,1}}}{1-e^{-x_{Z,1}}}\cdot \prod _{i=1}^2 \frac {1+\lambda^{-n_{Z,i}}\cdot e^{-y_{Z,i}}}{1-\lambda^{-n_{Z,i}}\cdot e^{-y_{Z,i}}}\right ]_Z\end{equation}
$$=4\cdot \left (\frac {1+\lambda ^{n_{Z,2}}}{1-\lambda ^{n_{Z,2}}}\cdot \frac {\lambda ^{n_{Z,1}}}{(1-\lambda ^{n_{Z,1}})^2}   \cdot [y_{Z,1}]_Z  + \frac {1+\lambda ^{n_{Z,1}}}{1-\lambda ^{n_{Z,1}}}\cdot \frac {\lambda ^{n_{Z,2}}}{(1-\lambda ^{n_{Z,2}})^2}   \cdot [y_{Z,2}]_Z \right ).$$

\bigskip
\noindent
By homotopy invariance the $S^1$-equivariant signature is rigid, i.e. constant in $\lambda$.
 Hence, the sum $\sum _{Z\subset M^{S^1}} \mu _Z(\lambda )$ does not depend on $\lambda $
  (this can be shown also by comparing both sides of (\ref{eq:19}) and observing that poles of the left hand side can only occur in $0,\infty$,
   whereas a pole of the right hand side must be on the unit circle).
    In particular, in view of equation (\ref{eq:20}) $S^1$ cannot act on $M$ with only one fixed
point.

Recall from the beginning of the appendix that all $n_{Z,i}$ are
positive. Taking the limit $\lambda \to \infty$ in the right hand
side of (A.7) one sees that
\begin{equation}\label{eq:22}sign (M)=\sum_{Z\subset M^{S^1}} sign (Z).\end{equation}

  \end{appendix}

\bibliography{comp}{}
\bibliographystyle{amsplain}
\end{document}